\documentclass{article}
\usepackage{amssymb,amsmath,amsfonts,amsthm,graphicx,ifthen,mathrsfs,hyperref}
\usepackage[mathcal]{euscript}
\usepackage{pgf,tikz}
\usetikzlibrary{arrows}

\usepackage[cmtip,arrow]{xy}
\usepackage{comment}
\usepackage{mathtools}

\theoremstyle{plain}
\newtheorem{Theorem}{Theorem}   [section]
\newtheorem{Lemma}[Theorem]   {Lemma}

\newtheorem{Definition}[Theorem] {Definition}

\newtheorem{Example}{Example}

\begin{document}

\title{Metric dimension of combinatorial game graphs}
\author{Craig Tennenhouse\\{\small University of New England, Biddeford, ME 04005, USA}\\ {\small ctennenhouse@une.edu}}
\date{}

\maketitle

\begin{abstract}
The study of combinatorial games is intimately tied to the study of graphs, as any game can be realized as a directed graph in which players take turns traversing the edges until reaching a sink. However, there have heretofore been few efforts towards analyzing game graphs using graph theoretic metrics and techniques.

A set $S$ of vertices in a graph $G$ \emph{resolves} $G$ if every vertex in $G$ is uniquely determined by the vector of its distances from the vertices in $S$. 	A \emph{metric basis} of $G$ is a smallest resolving set and the \emph{metric dimension} is the cardinality of a metric basis. In this article we examine the metric dimension of the graphs resulting from some rulesets, including both \emph{short} games (those which are sure to end after finitely many turns) and \emph{loopy} games (those games for which the associated graph contains cycles).  
\end{abstract}

\section{Introduction}\label{sec:intro}

A \emph{combinatorial game} is a turn-based game with two players wherein the last player to make a legal move wins. Its study, Combinatorial Game Theory (CGT), began with Bouton in \cite{bouton1901nim} but its modern foundations begin with Conway, Berlekamp, and Guy\cite{berlekamp2018winning}, and later \cite{albert2007lessons,siegel2013combinatorial}. 

The language of CGT differs slightly from colloquial language surrounding games. A \emph{ruleset} is a collection of rules governing game play, (e.g. \textsc{chess}, \textsc{checkers}, and \textsc{tic-tac-toe}), while a \emph{game} is a particular ruleset paired with a position. However, we will often use the terms \emph{game} and \emph{ruleset} interchangeably when the meaning is clear from context. We include a few necessary terms in this section.

A combinatorial game in which players have different move options, pieces, and/or goals is a \emph{partisan} game. A game that is not partisan is \emph{impartial}.

Length of play is another way to categorize a ruleset. A game is \emph{short} if it always ends within a finite number of turns. A game is \emph{loopy} if a position can be reached at more than one point throughout play.

The most well-known impartial short combinatorial game is \textsc{nim}. In \textsc{nim} players are presented with a multiset of heap sizes and take turns removing any positive number of stones from any one heap. The player who removes the final stone is the winner. Bouton \cite{berlekamp2018winning} determined the complete solution. 

Throughout we will use common notation and language conventions found in \cite{west1996introduction}. In particular, the following definitions will be useful in our discussion.

\begin{Definition}\label{def:simplegraph}
A \emph{graph} $G$ is an ordered pair $(V,E)$ where $E$ is a subset of the Cartesian product $V\times V$. If $G$ is \emph{simple} then $E$ is a set of unordered pairs of distinct elements, and if $G$ is a \emph{digraph} then the elements are ordered. An \emph{oriented graph} is a digraph in which $(u,v) \in E \Rightarrow (v,u)\notin E$. The \emph{underlying graph} of an oriented graph is the result of unordering the ordered pairs in $E$. An oriented graph is a \emph{tournament} if for all pairs $u,v\in V$ either $(u,v)$ or $(v,u)$ is in $E$. A \emph{source} is a vertex with no edges directed into it and a \emph{sink} has no edges directed away. We sometimes simplify notation by writing $v\in G$ and $e\in G$ instead of $v\in V(G), e\in E(G)$, respectively. 

The \emph{order} $n(G)$ of a graph $G=(V,E)$ is the cardinality of $V$, and its \emph{size} $e(G) = |E|$. The \emph{distance} $d(u,v)$ between $u$ and $v$ in a simple graph is the number of edges in a shortest path between the vertices. 
\end{Definition}

In a directed graph a path must be directed. Therefore, even if the underlying graph of a directed graph is connected, it's possible that $d(u,v)=\infty$; that is, there is no directed $uv$-path. Since the graphs that we will address which are associated with combinatorial games are often not strongly-connected, we introduce the following convention.

\begin{Definition}\label{def:signeddist}
In a digraph the \emph{signed distance} is defined in the following way:\\
$$
d_{\pm}(u,v) = 
\begin{cases}
	d(u,v) & \text{if } \exists \text{a } uv\text{-path and } d(u,v)\leq d(v,u) \\
	-d(v,u) & \text{if } \exists \text{a } vu\text{-path and } d(v,u)<d(u,v) \\
	\infty & \text{if } \text{no } uv\text{-path nor } vu\text{-path } exists
\end{cases}
$$
\end{Definition}

Since this convention will be used throughout we will simply use $d(u,v)$ to denote both regular and signed distance as context requires.

	
Often while discussing the play of a game, players wish to refer to a move as being within $n$-many moves of completion, e.g. \emph{``Checkmate in three.''} This motivates us to examine the simplest way to describe positions of a game using distances to other ``landmark'' positions. It is this endeavor which motivates us to study the metric dimension of game graphs.
	
Slater \cite{slater1975leaves} and Harrary \& Melter \cite{harary1976metric}	introduced the following concepts, which we have modified slightly for our purposes. 

\begin{Definition}\label{Resolving set}
In a graph $G=(V,E)$ with ordered subset $S = (u_1,u_2,\ldots ,u_k) \subseteq V$, the \emph{distance vector} of a vertex $v\in G$ is the ordered set of distances from vertices in $S$ to $v$. That is, 
$$d(S|v) = (d(u_1,v),d(u_2,v),\ldots ,d(u_k,v))$$
If every vertex in $G$ is uniquely determined by its distance vector then $S$ is a \emph{resolving set} for $G$, or $S$ \emph{resolves} $G$.

A smallest resolving set in $G$ is a \emph{metric basis}, and the cardinality of a metric basis is the graph's \emph{metric dimension}, denoted $\beta(G)$. If a set of vertices $V$ share a common distance vector via set $S$ then we say that $S$ does not \emph{disambiguate} $V$, or that $V$ is \emph{ambiguous} with respect to $S$.
\end{Definition}

There has been considerable work on determining the metric dimensions of various families of graphs, including wheels \cite{sooryanaranyana2002metric} and hypercubes \cite{erdos1963two,beardon2013resolving}. In addition, it has been shown in \cite{foucaud2015algorithms} that determining the metric dimension of graphs in general is computationally hard.

In this paper we address connections between combinatorial games and metric dimension. In Section~\ref{sec:gamegraphs} we demonstrate how graphs associated with combinatorial games can be constructed using various rulesets. In Section~\ref{sec:nim} we determine the metric dimensions of graphs associated with \textsc{nim}. Another short combinatorial game is addressed in Section~\ref{sec:takeaway}. Some loopy games are examined in Section~\ref{sec:loopy}, and we conclude with some open questions on metric dimension and combinatorial game graphs in Section~\ref{sec:conclusion}.

\section{Game graphs}\label{sec:gamegraphs}

Consider a short combinatorial ruleset \textsc{R} and a starting position $\mathcal{G}$. Traditionally the graph associated with \text{R} is a tree directed from the root $\mathcal{G}$ and with vertices associated with positions, where players take turns traversing one edge at a time until reaching a leaf vertex. We modify this associated graph slightly and say that the \emph{game graph} $G($\textsc{R},$\mathcal{G})$ is the digraph resulting from identifying vertices in this oriented tree that are associated with the same position. If the ruleset is obvious we omit \textsc{R} and simply write $G(\mathcal{G})$. In the case of partisan games we ignore partisan options and consider options available to either player as available to both.

For example, the reader is likely aware of the impartial game \textsc{nim}, wherein players take turns removing any positive number of stones from any single heap until all heaps are empty. A single \textsc{nim} heap with $n$ stones is denoted $*n$. The game graph associated with a single \textsc{nim} heap of size $3$, $G(\textsc{nim},*3)$, is the acyclic tournament on $4$ vertices $\{0,*,*2,*3\}$. 

Note that game graphs associated with short games will be acyclic, and for loopy games will contain either loops or cycles.

For some initial examples let us consider some combinatorial games in canonical form. 

\begin{Definition}\label{def:cform}
A game is in \emph{canonical form} if there are no dominated options and no reversible moves.  
\end{Definition}
That is, in a game noted $\{L|R\}$ where $L$ is the set of Left's options and $R$ the set of Right's, $L$ and $R$ are as simple as possible in order to contain all important information about a game.

\begin{Example}
Recall that under standard conventions of combinatorial game theory, the canonical form of the game $n \in \mathbb{Z}\setminus \{0\}$ is $\{n-1|\}$, and the game $-n = \{|-n+1\}$. So $G(-n) \cong G(n)$ is a directed path on the $n+1$ vertices $\{0,1,2,\ldots ,n\}$.
\end{Example}

\begin{Example}
The game $\{x|y\}$ with $x<y$ has value equal to $\frac{p}{2^n}$ for integers $p,n$ with $2\nmid p$, where $\frac{p}{2^n}$ is the simplest dyadic rational between $x$ and $y$. That is, $n$ is the least positive integer such that $x<\frac{p}{2^n}<y$. The game graph $G(\frac{p}{2})$ is a transitive tournament on three vertices with a directed path graph attached.  
\end{Example}




The graphs of other dyadic rationals can similarly be computed, along with other canonical form games including infinitessimals and switches.

It is worth noting that not all graphs admit an orientation that is the game graph of some canonical form game. As an example, consider any graph $H$ with three vertices $u,v,w$, of degree one. Any orientation of $H$ with a single source (as no game graph can have more than one source) will therefore result in at least two sinks. Since the only sink in a canonical form game is the position $0$, such a graph is impossible. However, there are a number of rulesets for which a game graph may contain a number of distinct sinks, representing different game positions with game value $0$.

\section{Metric dimension of Nim}\label{sec:nim}

We now consider game graphs for the game \textsc{nim}, as described in Section~\ref{sec:intro}. We will denote by \textsc{nim}$(H)$ the game of \textsc{nim} with heap sizes given by the multiset $H=\{h_1,h_2,\ldots ,h_k\}$ . For simplicity we will denote this graph by $G(h_1,h_2,\ldots ,h_k)$ and assume that $h_i\leq h_j \forall i<j$. 

First we examine the single-heap game. Recall that a single heap is denoted $*n$ so we simplify the graph notation further to $G(*n)$.

\begin{Theorem}\label{thm:nim1}
For $n\in \mathbb{Z}, n\geq 1$ the graph $G(\textsc{nim},*n)$ has order $n+1$, size $\binom{n+2}{2}$, and metric dimension $\lceil \frac{n}{2}\rceil$.
\end{Theorem}
\begin{proof}
The set of vertices in $G(*n)$ is $\{0,1,\ldots ,n\}$, and since every position can reach every smaller position in a single move the graph $G(*n)$ is the acyclic tournament on $n+1$ vertices. Thus its size is $\binom{n+2}{2}$. 

If $S\subseteq V(G(*n))$ contains fewer than $\lceil \frac{n}{2}\rceil$ vertices then there is an $i$ so that neither $*i$ nor $*(i+1)$ are in $S$. Both of these vertices are signed distance $1$ from every vertex associated with a larger heap, and signed distance $-1$ from each smaller heap. Thus $G(*n)$ is not resolved by $S$. However, the set $S' = \{*(2k-1) : 1\leq k\leq \frac{n}{2}\}$ mitigates this issue and resolves the graph.
\end{proof}

We note that between the game graphs of integers we saw in Section~\ref{sec:gamegraphs} and game graphs of single-heap \textsc{nim} we achieve both possible extreme values of connected graph density, $n / \binom{n+1}{2}$ and $1$, respectively.


Now we consider games with two heaps. While the metric dimension of $G(H)$ is computable for every multiset $H$, we will restrict ourselves to only the diagonal positions, i.e. those graphs of the form $G(n,n)$ for some positive integer $n$, since every \textsc{nim} position is reachable from infinitely many diagonal positions. 

\begin{Theorem}\label{thm:nim2}
The graph $G=G(n,n)$ has order $\frac{1}{2}(n^2+3n+2)$, size $\frac{1}{2}n(n+1)^2$, and metric dimension $\lceil \frac{n}{2}\rceil$.
\end{Theorem}
\begin{proof}
The order is simply the combinations of possible heap sizes, of which there are $n+1$, with replacement, resulting in $\binom{n+2}{2} = \frac{1}{2}(n^2+3n+2)$. The size of the graph is easily computed by summing the out-degrees of all vertices. Every vertex associated with a diagonal position $\{k,k\}$ has out-degree $k$, and each one associated with a non-diagonal position $\{a,b\}$, $a<b$ has out-degree $a+b$. The total out-degree sum is 
$$\sum\limits_{0\leq a<b\leq n} (a+b) + \sum\limits_{k=1}^{n}k $$
which yields the value we want.

\begin{figure}
\centering

\includegraphics[width=.5\textwidth]{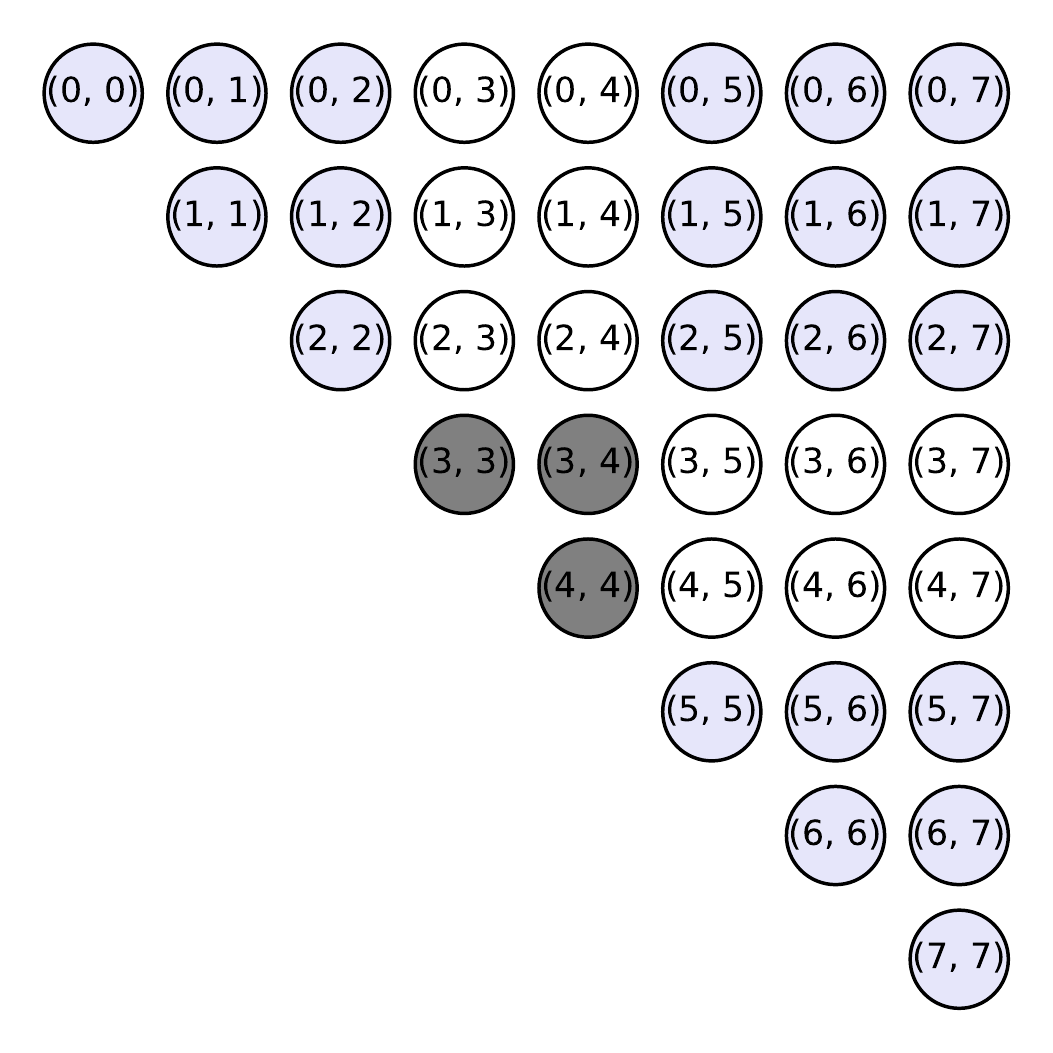}
\caption{Vertices of the game graph $G(\textsc{nim},(7,7))$ with $G', G'', G'''$ highlighted}\label{fig:nim2}
\end{figure}

Say that $S$ resolves the graph $G$ and neither $(x,x)$ nor $(x+1,x+1)$ is in $S$. Let $G'$ be the graph induced by the vertices representing games with both heaps smaller than $x$, and let $G''$ be the graph induced by the vertices of the form $(a,b)$ with $a\notin \{x,x+1\}$ and $b>(x+1)$ (Figure~\ref{fig:nim2}). No set of vertices in $G'$ disambiguates $\{(x,x),(x,x+1),(x+1,x+1)\}$ since the signed distance from each vertex in $G'$ is $-2$. Nor does any vertex in $G''$ disambiguate these vertices since all three have the same signed distance from each vertex in $G'$. If $v\in S$ of the form $(a,x),a<x$ then $v$ has the same distance to $(x,x)$ and $(x,x+1)$, and any vertex of the form $(a,x+1)$ has the same distance to $(x,x+1)$ and $(x+1,x+1)$. Similarly, if $b>(x+1)$ then $(x,b)$ does not disambiguate $(x,x),(x,x+1)$, nor does $(x+1,b)$ disambiguate $(x,x+1),(x+1,x+1)$. Therefore, either $(x,x+1)$ or at least two vertices in $G\setminus \{G',G''\}$ are in $S$. Since each vertex in $G\setminus \{G',G''\}$ represents two distinct heaps, each can be used at most twice to disambiguate diagonal positions. Hence $|S|\geq \lceil \frac{n}{2}\rceil$.

Now consider $S=\{(k,k):1\leq k\leq n \text{ odd}\}$. Then for $(x,x)\in S, (a,b)\in G$,
$$
d((x,x),(a,b)) = 
\begin{cases}
	0 & \text{if } a=b=x \\
	1 & \text{if } a<b=x \\
	-1 & \text{if } x=a<b \\
	2 & \text{if } a<b=x \\
	-2 & \text{if } x<a<b
\end{cases}
$$

If both $a$ and $b$ are odd then they are uniquely determined by the presence of $1$ and/or $(-1)$ in the distance vector. If both are even then their value is clear from the indices associated with the indices of $2$ and/or $(-2)$. If one is odd and the other even then, similarly, some combination of these values uniquely identifies the vertex $(a,b)$. Hence $S$ resolves $G$.
\end{proof}

	
%
%
%
%
	
\section{Another takeaway game}\label{sec:takeaway}
	
	
Let's examine a \textsc{nim} variant.

\begin{Definition}\label{def:hats}
Let $H$ be a multiset of positive integers. In \textsc{hats-nim} (\textbf{H}eaps \textbf{A}re \textbf{T}he \textbf{S}ubtraction sets) players take turns subtracting exactly $s$ from any one heap in $H$, where $s\in H$. 
\end{Definition}	

In \textsc{hats-nim} every single-heap game has a single move to $0$, hence the metric dimension of the associated graph is $1$ with either vertex comprising a metric basis. In Theorem~\ref{thm:hats} we determine the metric dimension of the associated graphs of all two-heap games. See Figure~\ref{fig:hats} for an example of the resulting game graph.

\begin{figure}
\centering
\hspace{1in} \includegraphics[width=.7\textwidth]{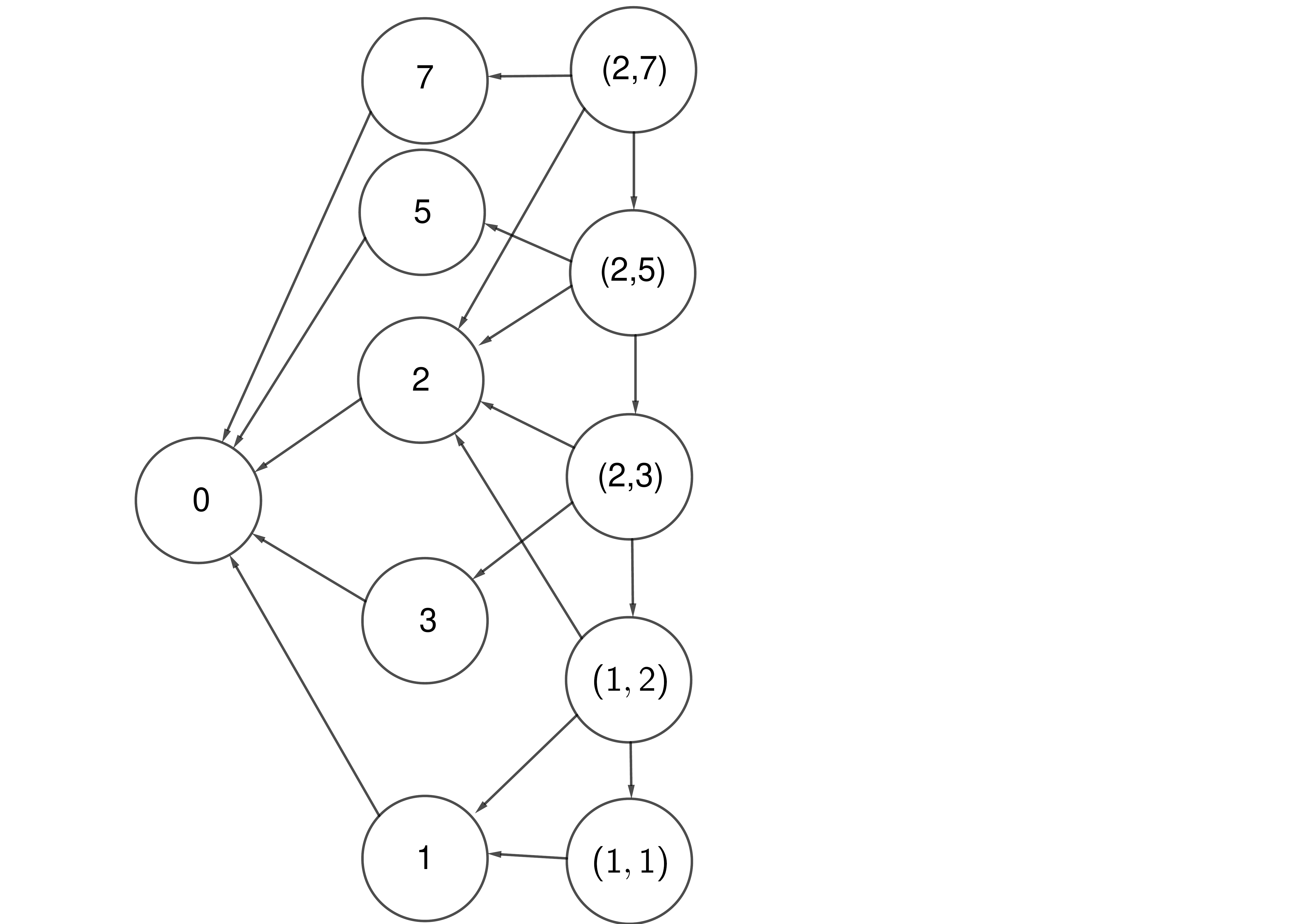}
\caption{The graph $G(\textsc{hats-nim},(2,7))$}\label{fig:hats}
\end{figure}

\begin{Theorem}\label{thm:hats}
Let $H = \{a,b\}$ be a two-heap multiset, with $0<a\leq b$, and let $G=G(\textsc{hats},H)$ be the associated game graph. The metric dimension of $G$ is 
$$ \beta(G)=
\begin{cases}
	1 & \text{ when } a=b \text{ or } H=\{1,2\}\\
	2 & \text{ when } 2<b=(a+1) \text{ or } 2<a=1\\
	3 & \text{ otherwise}
\end{cases}
$$
\end{Theorem}
\begin{proof}
We need only consider games in which the greatest common divisor of $a$ and $b$ is $1$, since otherwise the game is equivalent playing on $\{\frac{a}{\gcd(a,b)},\frac{b}{\gcd(a,b)}\}$. 

First we consider when $a=b$ or $H=\{1,2\}$. In the former case the graph $G$ is simply a directed graph on three vertices, $\{a,a\} \rightarrow \{a\} \rightarrow 0$, which is resolved by any single vertex. In the latter case the graph is easily seen through inspection to be resolved by $\{1,1\}$, as the four other vertices have distance $1,2,-1$, and $\infty$ from $\{1,1\}$.

Now assume that $b=(a+1)$ and $a>1$. Any two-heap position has two single-heap vertices at distance $1$, and any single-heap position has two other single-heap positions at distance $\infty$. So if $|S|=1$ then $S$ does not resolve $G$. Instead let $S=\{\{1\},\{1,a+1\}\}$. If $\{x,y\}, x\leq y$ is a position with $x=0$ then the distance from $\{a,a+1\}$ is equal to $y$. If $x>0$ then $y=(x+1)$. As there is a single path from $\{1,a+b\}$ containing only two-heap positions, the distance from $\{1,a+1\}$ identifies $\{x,y\}$. Hence $S$ resolves $G$.

If $a=1,b>2$ then the same argument as in the previous case demonstrates that any resolving set must contains at least two vertices. Let $S=\{\{1\},\{1,b\}\}$, just as above. In the same way, we see that $S$ resolves $G$. 

Finally we assume that $b>a+1$ and $a$ is at least $2$. Assume that $|S|=2$. There are at least four single-heap vertices in $G$: $\{a\},\{b\},\{b-a\}$, and $\{|b-2a|\}$. If $S=\{\{x\},\{y\}\}$ then at least two single-heap games remain in $S^c$ and are ambiguous. If $S=\{\{x\},\{x,y\}\}$ or $\{\{y\},\{x,y\}\}$ with $x>1$ then $\{1\}$ and $\{1,1\}$ remain ambiguous to $S$, and if $x=1$ then there are at least two single-heap positions greater than $y$. These single-heap positions are not disambiguated by $S$. 

If $S = \{\{x\},\{y,z\}\}$ with $x\notin \{y,z\}$ then $\{y\}$ and $\{z\}$ are both $\infty$ from $\{x\}$ and $1$ from $\{y,z\}$. Finally, if $S=\{\{x,y\},\{z,w\}\}$ then $\{1\}$ and $\{1,1\}$ are not disambiguated by $S$ unless $\{1,1\} \in S$, in which case $\{z\}$ and $\{w\}$ are ambiguous. Note that $\{1,1\}$ is the only diagonal position in the game since $gcd(a,b)=1$.
\end{proof}	
	

%
%
%
%

\section{Metric dimension of some loopy games}\label{sec:loopy}

Recall from Section~\ref{sec:intro} that a game is \emph{loopy} if a single position can arise more than once during play. This is equivalent to its associated game graph containing cycles or loops. There are many ways that rulesets can admit loops, from allowing players to move a piece back to a previous position to permitting players to pass their turn. Some loopy games allow only one player to pass or return to a previous position. These games are referred to as \emph{OSLo} games (\textbf{O}ne-\textbf{S}ided \textbf{Lo}opy) \cite{siegel2011structure}.

Some of the simplest loopy games to visualize are variants of chess. 




Let $B(a,b)$ be the \emph{Bishop Graph} on an $a\times b$ grid. That is, represent each square in the grid by a vertex and adjoin a pair of vertices representing a legal bishop move in chess (any diagonal direction and any distance). 

\begin{figure}
\centering

\includegraphics[width=.2\textwidth]{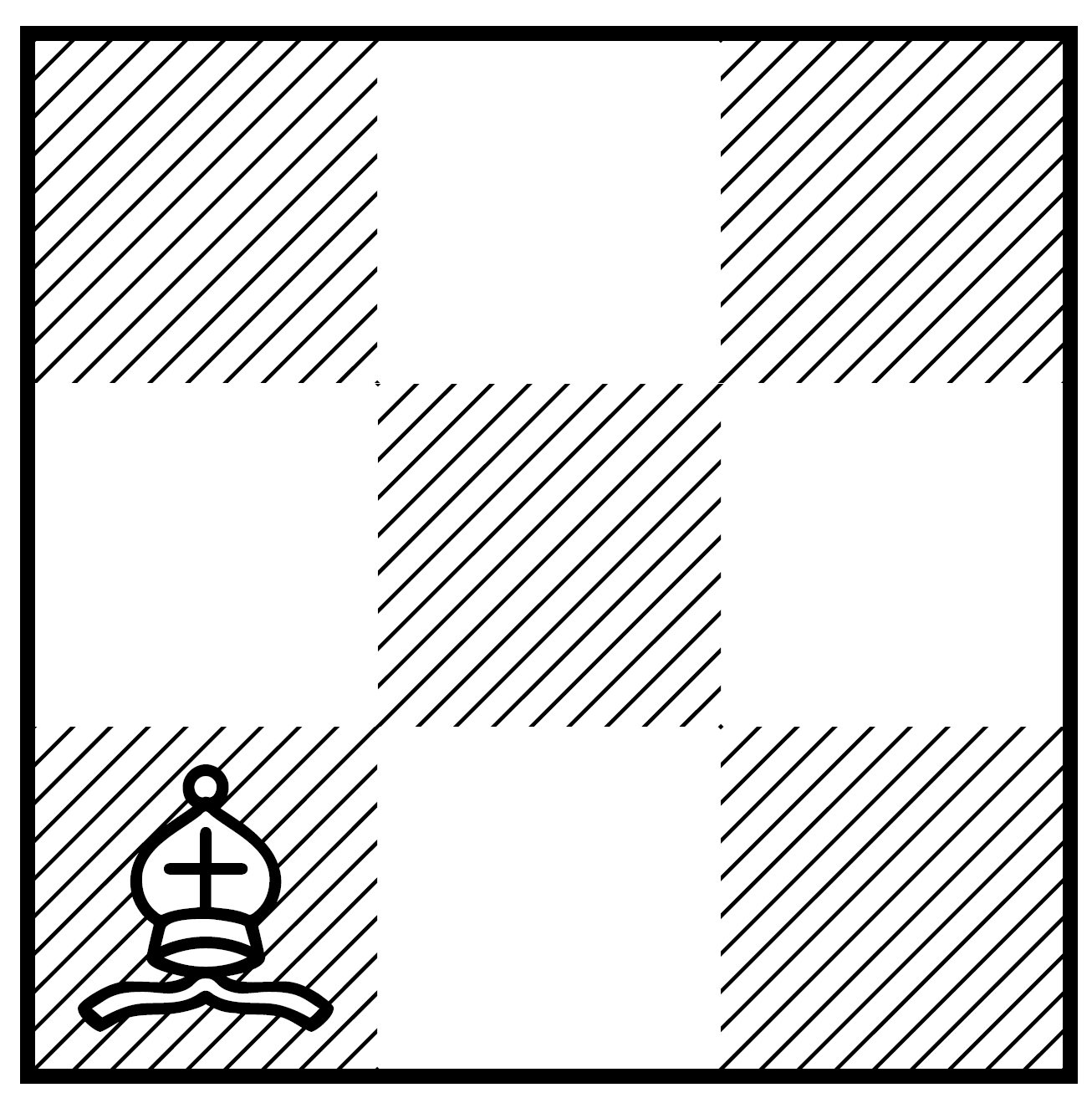}
\includegraphics[width=.2\textwidth]{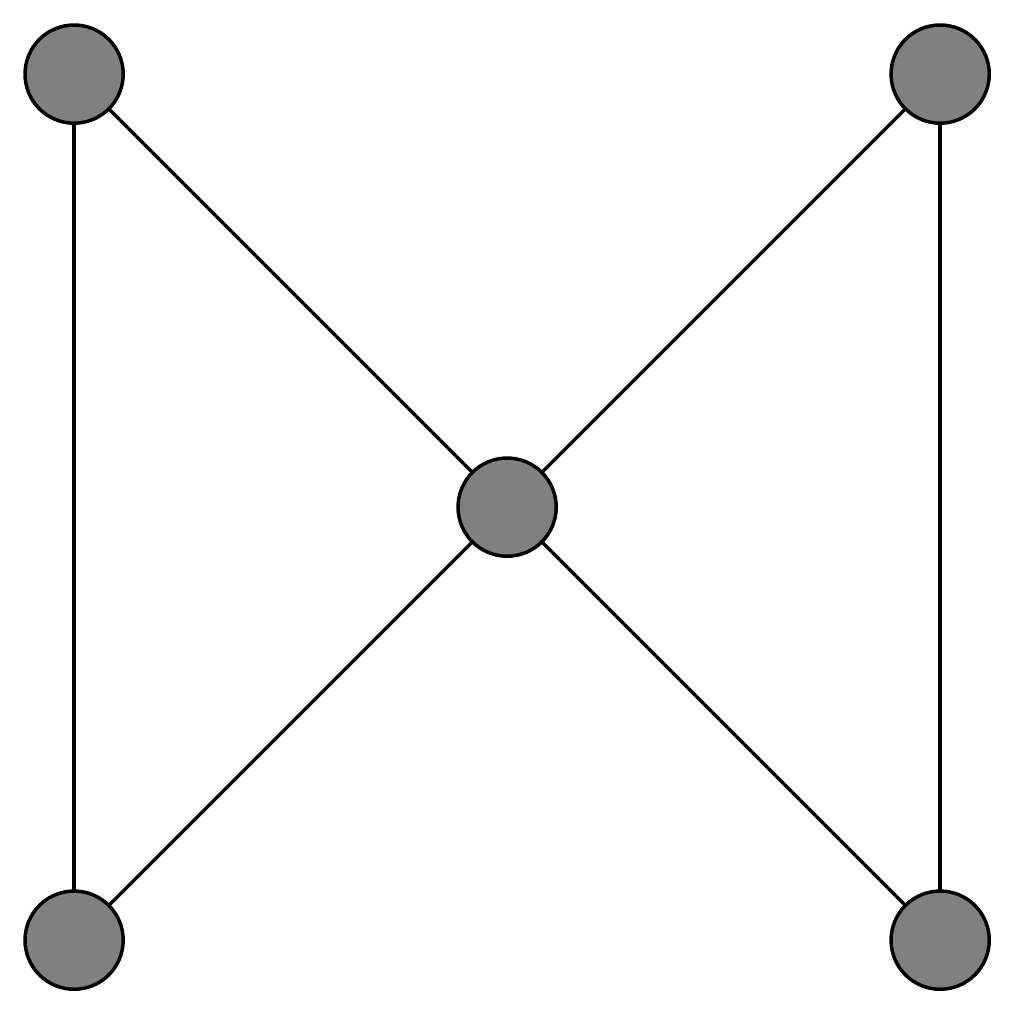}
\includegraphics[width=.2\textwidth]{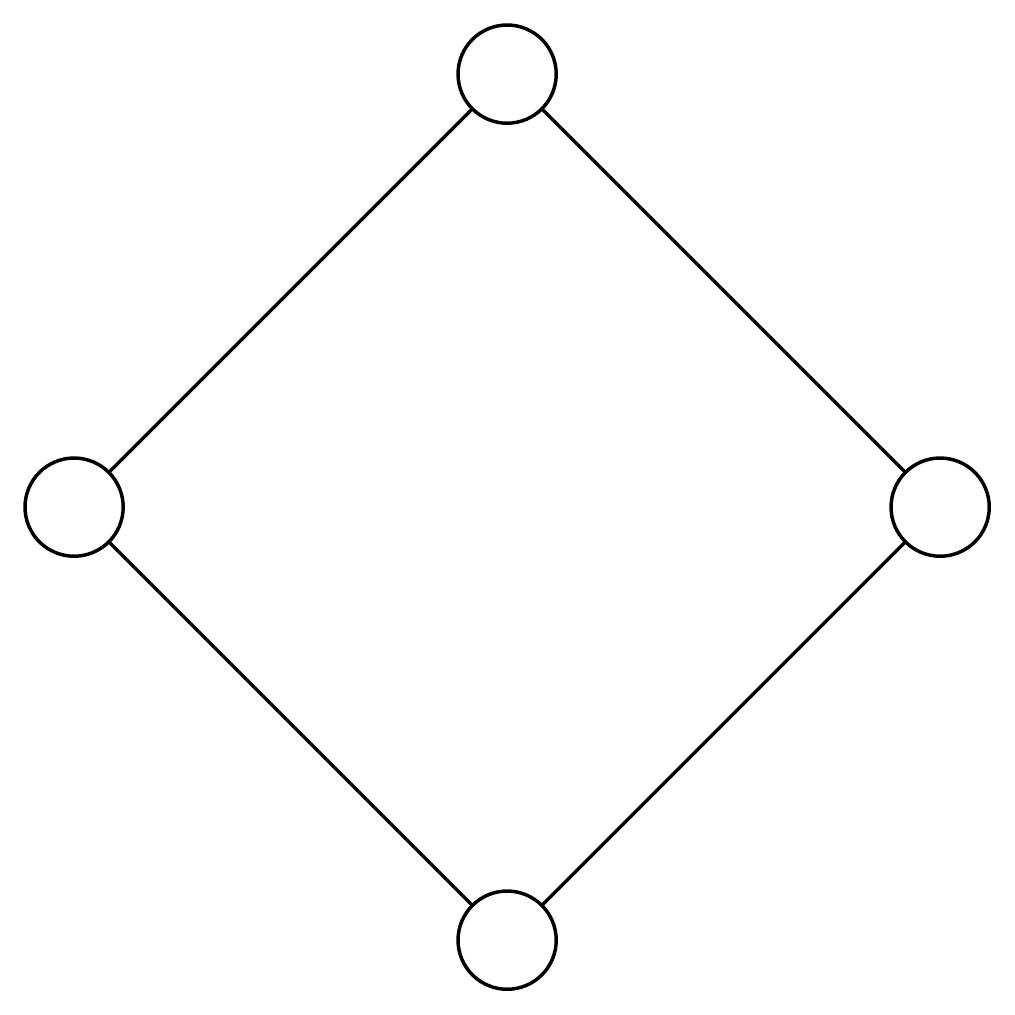}
\caption{The disconnected bishop graph on a $3\times 3$ grid}\label{fig:bishopgraph}
\end{figure}

Since a bishop on any one colored square can only move to other squares of that color the graph $B(a,b)$ is disconnected, with two distinct connected components (Figure~\ref{fig:bishopgraph}). We refer to these components as $B_L(a,b)$ and $B_D(a,b)$ for the light and dark squares, respectively. If we place a single bishop on one square of each color then neither will threaten the other. We refer to this position as $B_Lb_D$. The game graph $G(B(a,b),B_Lb_D)$ is therefore the cartesian product $B_L(a,b) \square B_D(a,b)$, with each edge replaced by a pair of symmetric arcs. Thus $G(B(a,b),B_Lb_D)$ has order equal to $\lceil \frac{ab}{2}\rceil\lfloor \frac{ab}{2} \rfloor$. For simplicity we will ignore these symmetric arcs and instead simply consider the simple underlying graph. 

Before examining $G(B(a,b))$ we note a Lemma from \cite{caceres2007metric}.

\begin{Lemma}\label{lem:cartprod}
If $\beta(G\square H)=2$ then one of the graphs $G,H$ is a path graph.
\end{Lemma}

Now we determine the metric dimension of some bishop graphs.

\begin{Theorem}\label{thm:bishop}
For the game graph $G(B(a,b),B_Lb_D))$,
\begin{enumerate}
	\item If $n>1$ then $\beta(G(B(2,n),Bb)) = 2$. 
	\item $\beta(G(B(3,3),B_Lb_D)) = \beta(G(B(3,4),B_Lb_D)) = 3$
\end{enumerate}
\end{Theorem}
\begin{proof}
We note that both $B_L(2,n)$ and $B_D(2,n)$ are path graphs for all $n>1$. Therefore their cartesian product is a grid graph which has metric dimension $2$.

Note that neither $B_L(a,b)$ nor $B_D(a,b)$ is a path graph for $a,b>2$. Hence by Lemma~\ref{lem:cartprod} $\beta(G(B(a,b),B_Lb_D))>2$. First we consider $G=G(B(3,3),B_Lb_D)$. One of $B_D(3,3), B_L(3,3)$ is isomorphic to $C_4$ and the other to the star on $5$ vertices, so we can consider $G\cong K_{1,4}\square C_4$. 

Consider the set $S = \{((0,0),(0,1)),((0,0),(1,0)),((0,2),(0,1))\}$. A case analysis (omitted but repeatable by inspection) shows that $S$ resolves $G$.

Similarly, if $G' = G(B(3,4),B_Lb_D)$ then $B_L(3,4)\equiv B_D(3,4)$, and the set $S = \{((0, 1), (0, 0)), ((1, 0), (1, 3)), ((2, 3), (1, 3))\}$ resolves the graph $G'$.
\end{proof}

We can also bound the value of $\beta(G(B(3,n),B_Lb_D))$ using doubly resolving sets.

\begin{Definition}\label{def:dres}
A set $S$ of vertices in a graph $G$ \emph{doubly resolves} the vertices $u,v$ if there is a pair $s_1,s_2\in S$ such that $d(s_1,u)-d(s_2,u) \neq d(s_1,v)-d(s_2,v)$. The set $S$ is a \emph{doubly-resolving set} for the graph $G$ if every pair of vertices in $G$ is doubly-resolved by a pair of vertices in $S$.
\end{Definition}

We denote the cardinality of a smallest doubly-resolving set in $G$ by $\psi (G)$. Lemma~\ref{lem:dres} was also proven in  \cite{caceres2007metric}.

\begin{Lemma}\label{lem:dres}
$\beta(G\square H) \leq \beta(G) + \psi(H) - 1$
\end{Lemma}


Since the game graph with non-attacking bishops is equivalent to a simple graph, this leads us to Theorem~\ref{thm:b3xn} which bounds the metric dimension of $3\times n$ two-bishop game graphs.

\begin{Theorem}\label{thm:b3xn}
$3\leq \beta(G(B(3,n),B_Lb_D)) \leq 4$
\end{Theorem}
\begin{proof}
First we note that the lower bound follows from Lemma~\ref{lem:cartprod}. For the upper bound consider $B_D(3,n)$ and $B_L(3,n)$. Depending on the parity of $n$, each is isomorphic to one of the graphs in Figure~\ref{fig:bsubs}. A metric basis of size $2$ for each is highlighted and is easy to verify, and we see doubly resolving sets of size $3$. Since $\beta(B_D(3,n))=\beta(B_L(3,n))=2$ and $\psi(B_D(3,n))=\psi(B_L(3,n))\leq 3$, Lemma~\ref{lem:dres} gives us an upper bound of $4$ on $\beta(B_D(3,n)\square B_L(3,n)) = \beta(G(B(3,n),B_Lb_D))$. 
\end{proof}

\begin{figure}
\centering

\includegraphics[width=.9\textwidth]{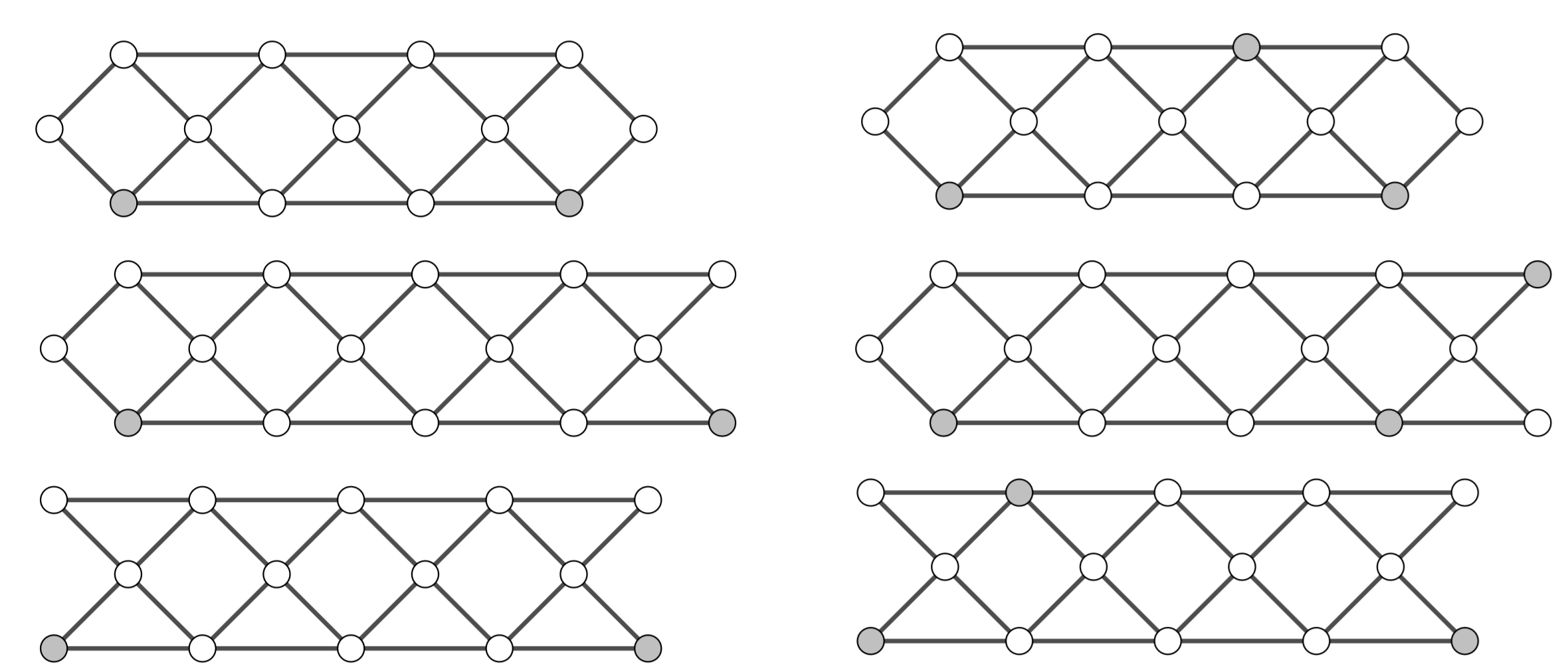}
\caption{All three possible connected components of a $3\times n$ bishop graph, with highlighted Metric Bases (left) and smallest Doubly Resolving Sets (right)}\label{fig:bsubs}
\end{figure}

Next we move on to another simplified chess variant. Recall that a king can move a distance of one square in any of the $8$ directions, and can attack in any of these squares. Define the \emph{King Graph} $K(a,b)$ in the obvious way and consider a pair of opposing kings on a $1\times n$ strip. Denote the position wherein the first king is on square $i$ and the second on square $j$ as $K_ik_j$, an example of a \emph{bare king} position. Since the game is loopy the game graph $G(K(1,n),K_ik_j)$ contains directed cycles, however it is not strongly connected since positions containing a single king have no directed path to positions containing both kings. Since we will consider the game graph to contain all possible positions on the strip we abbreviate the notation to $G(K(1,n))$. See Figure~\ref{fig:kstrip} for an example.

\begin{figure}
\centering
\includegraphics[width=.6\textwidth]{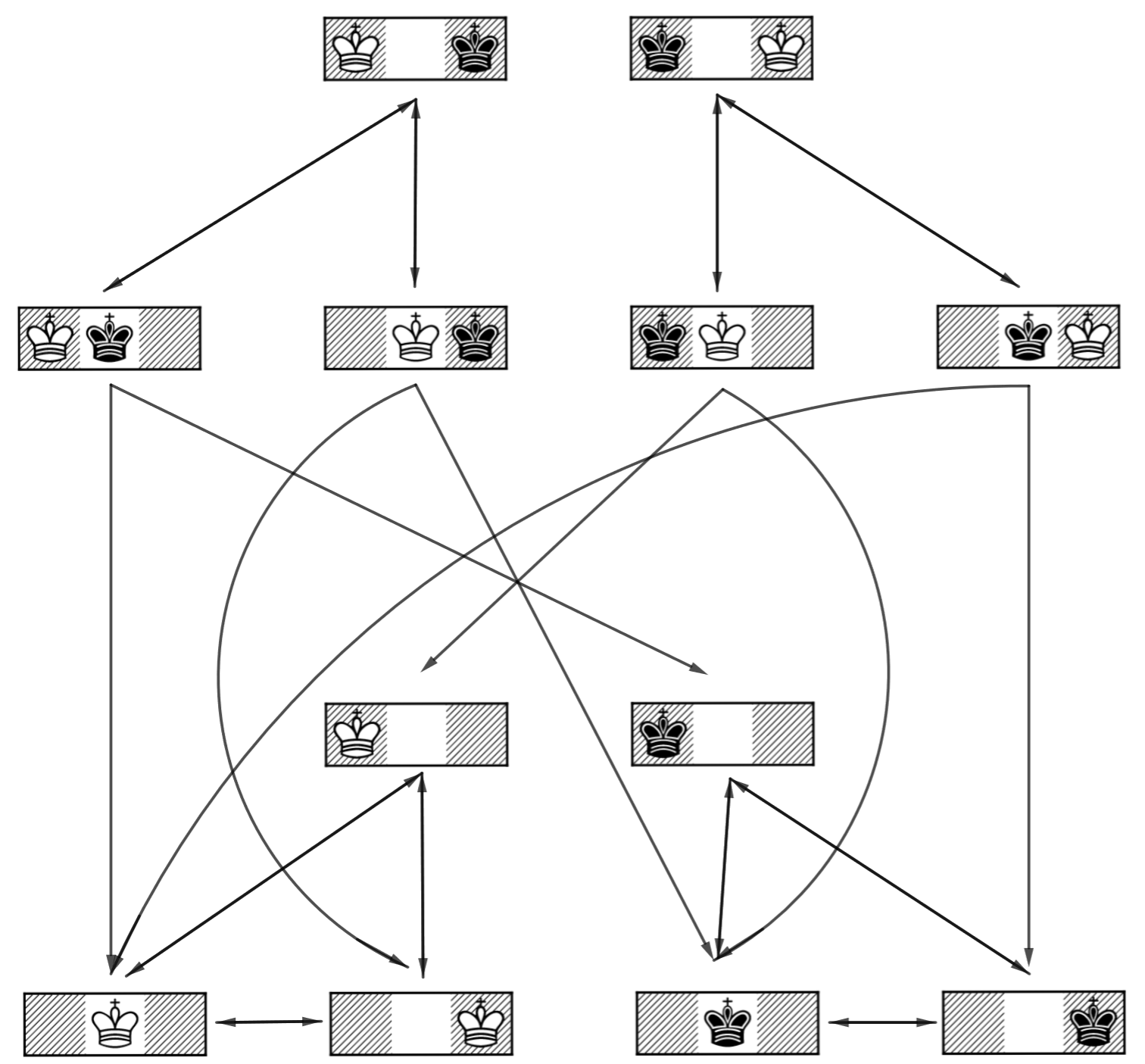}
\caption{The graph associated with the game with opposing kings on a $1\times 3$ strip}\label{fig:kstrip}
\end{figure}

\begin{Theorem}\label{thm:stripking}
The metric dimension of the game graph $G=G(K(1,n))$ is
$$ \beta(G) = 
\begin{cases}
	2 & \text{ if } 2\leq n\leq 3 \\
	3 & \text{ if } n>3
\end{cases}
$$
\end{Theorem}
\begin{proof}
The graph $G=G(K(1,2))$ has vertices $\{K_1k_2,K_2k_1, K_1,K_2,k_1,k_2\}$. Each vertex either has two out-neighbors, two in-neighbors, or at least two other vertices at distance $\infty$. So any resolving set must have at least $2$ vertices. $\{K_2,k_1\}$ suffices.

 If $n=3$ then $G=G(K(1,3))$ has order $12$ as seen in Figure~\ref{fig:kstrip}. As above, any position with a single king is at distance $\infty$ to all positions with only the other king, and any position with both kings has two positions at equal distance, each containing a single king. Thus a resolving set must contain at least $2$ vertices. Let $S = \{K_1,k_3\}$. Single-king positions are easily identifiable by their distances from $S$. The unique distance vectors are in Figure~\ref{table:stripking}.  

\begin{figure}
\centering
\begin{tabular}{|r|l|l|}
\hline
$v$ & $d(K_1,v)$ & $d(k_3,v)$ \\
\hline
$K_1k_2$ & $-2$ & $-3$ \\
$K_1k_3$ & $-3$ & $-3$ \\
$K_2k_1$ & $-1$ & $-2$ \\
$K_2k_3$ & $-3$ & $-2$ \\
$K_3k_1$ & $-2$ & $-2$ \\
$K_3k_2$ & $-2$ & $-1$ \\
$K_2$\hspace{.13in} & $1$ & $\infty$ \\
$K_3$\hspace{.13in} & $2$ & $\infty$ \\
$k_1$ & $\infty$ & $2$ \\
$k_2$ & $\infty$ & $1$ \\
\hline 
\end{tabular}
\caption{Table of distances from two landmark vertices in $G(K(1,2))$}\label{table:stripking}
\end{figure}

Now consider $n>3$ and assume that $S$ resolves $G=G(K(1,n))$. We know that $|S|>1$ from the argument above. If $|S|=2$ then $S$ must have two positions with both kings, $K_ik_j$ and $K_ak_b$ with $i<j,b<a$, as there are too many positions at distance $\infty$ for either position in $S$ to represent only one. But then each single-king position $K_c$ is reachable in the same distance as another single-king position $k_d$ from each of the positions in $S$. So $|S|>2$.

Now let $S=\{K_1,K_1k_2,K_2k_1\}$ and $v$ a vertex in $G$. If $d(K_1,v)=\infty$ then $v=k_i$ for some $i$ and the distance $d(K_1k_2,v)=i$. If $d(K_1k_2,v)=\infty$ then $v=K_ik_j,i>j$, $(i-1) = -d(K_1,v)$, and $j=d(K_1,v)+d(K_2k_1)+2$. If $d(K_2k_1)=\infty$ then $v=K_ik_j,i<j$. Here $j=d(K_1,v)$ and $i=d(K_1k_2) = -(i+j)+3$. Thus $S$ resolves $G$ and $\beta(G) = 3$.  
\end{proof}

Game graphs of this type can get quite large very quickly. As an example we present in Figure~\ref{fig:king2x2} the graph $G(K(2,2))$ and a metric basis without proof.

\begin{figure}
\centering
\includegraphics[width=.5\textwidth]{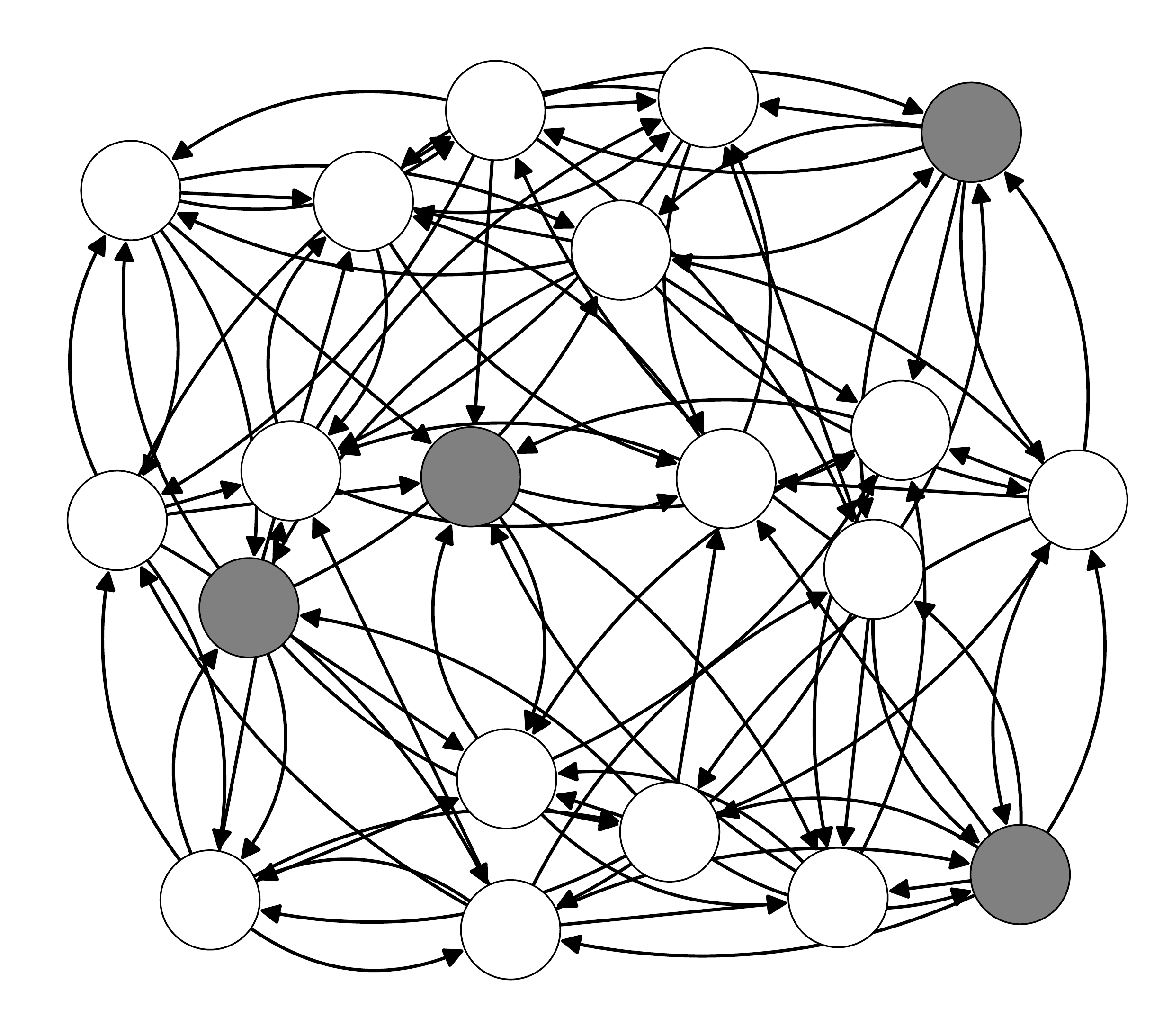}
\caption{The game graph associated with the game consisting of two opposing kings on a $2\times 2$ chessboard}\label{fig:king2x2}
\end{figure}

\textsc{Fox and Geese} is a commonly studied loopy game. Played on a checkerboard, the pieces are a single fox which begins at the bottom of the board and moves as a king in \textsc{checkers} (one space in any diagonal direction), and a set of geese which begin at the top and move as regular pieces in \text{checkers} (one space diagonally towards the bottom of the board). All pieces begin and remain on dark squares. Play proceeds until either the fox is trapped or the geese have exhausted their moves. For the most comprehensive CGT analysis see \cite{siegel2013combinatorial}. 

We can generalize this to an $x\times y$ grid and assume that the number of geese is equal to $g = \lceil \frac{x}{2} \rceil$, (see Figure~\ref{fig:foxgeeseboard}). The number of vertices in the associated game graph depends on the dimensions of the board. There are $s = \lceil \frac{xy}{2} \rceil$ usable squares. The fox is placed on one space and $\lceil \frac{y}{2} \rceil$ geese are placed on the remainder. Therefore the order of the game graph is $s \binom{s-1}{g}$. 

The game graph $G$ representing \textsc{Fox and Geese} on a $3\times 3$ board, shown in Figure~\ref{fig:foxgeeseboard}, has order $30$ and size $48$. One example of a metric basis, highlighted in the figure, is the positions given in the table in Figure~\ref{table:fg}.

\begin{figure}
\centering
\includegraphics[width=.3\textwidth]{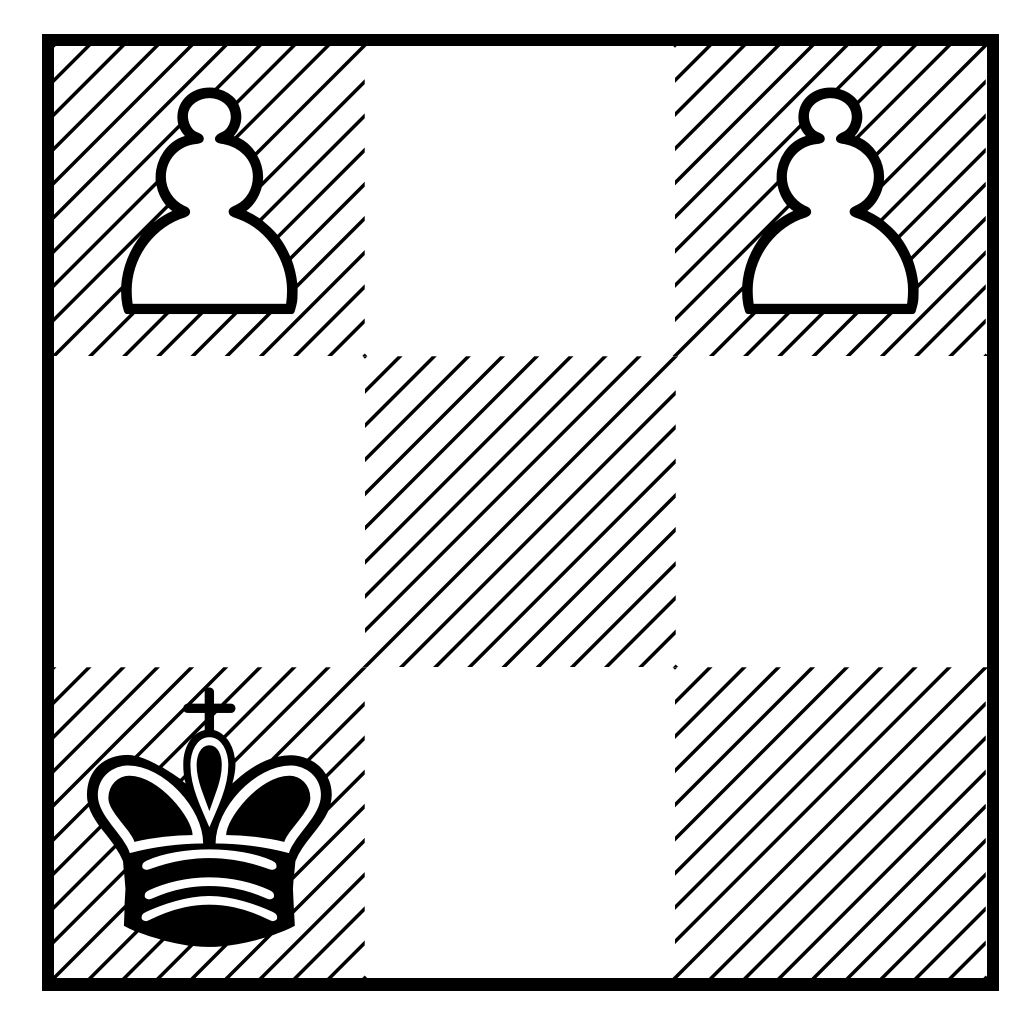}
\includegraphics[width=.4\textwidth]{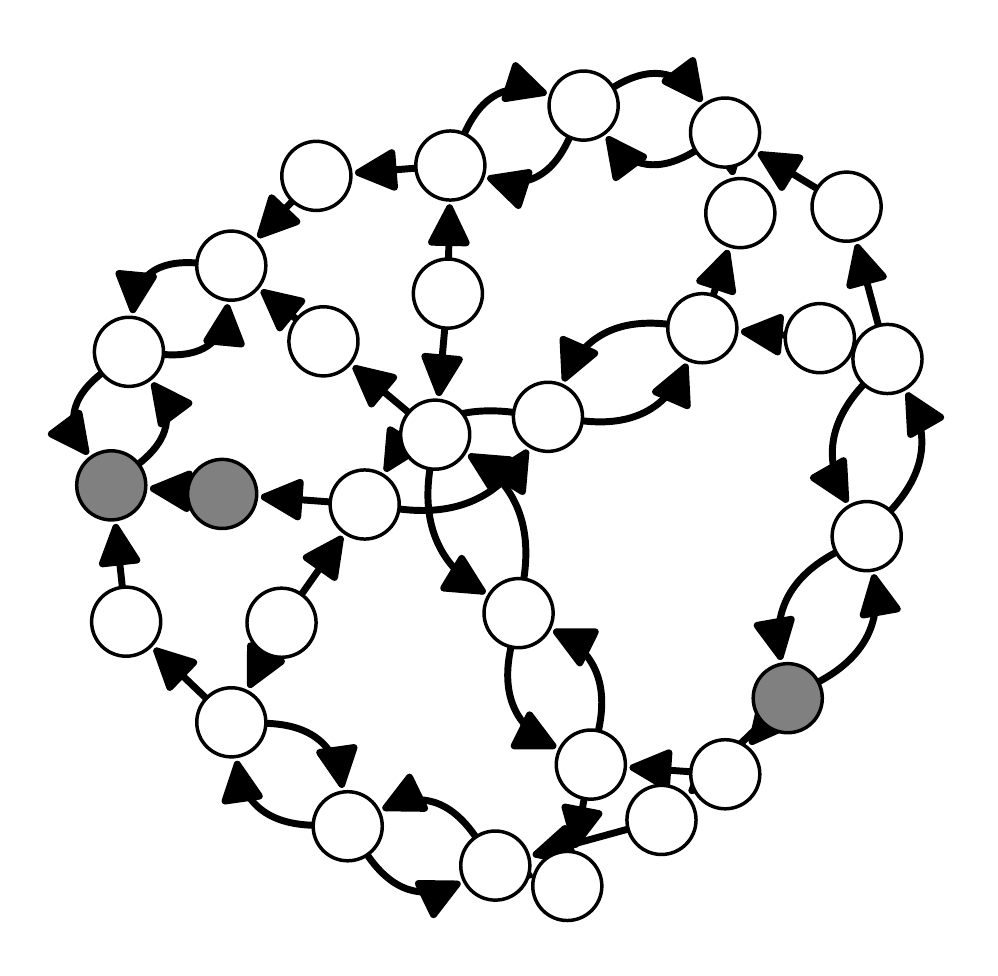}
\caption{Fox and Geese played with a king and pawns, respectively, on a $3\times 3$ board and the associated graph with a metric basis highlighted}\label{fig:foxgeeseboard}
\end{figure}

\begin{figure}
\centering
\begin{tabular}{|c|c|c|c|}
\hline
vertex & fox & goose $1$ & goose $2$ \\
\hline
$v_1$ & $(1, 3)$ & $(1, 1)$ & $(2, 2)$ \\  
$v_2$ & $(1, 3)$ & $(1, 1)$ & $(3, 1)$ \\  
$v_3$ & $(1, 1)$ & $(1, 3)$ & $(3, 3)$ \\
\hline
\end{tabular}
\caption{A metric basis for \textsc{fox and geese} on a $3\times 3$ board}\label{table:fg}
\end{figure}

\section{Other games and future directions}\label{sec:conclusion}

Many games that are not combinatorial can be analyzed in the manner this paper has described. For example, the positions in \textsc{tic-tac-toe} can be associated with vertices in a game graph although the game itself does not meet the ending conditions of CGT. Similarly the \textsc{towers of hanoi}, a solitaire game, has an associated graph that is rather well-studied. These game graphs are ripe for study.  

Within this paper we have only analyzed impartial games, whether traditionally impartial or modified to be so. It would be interesting to see an examination of partisan games, although yet another concept of graph distance would need to be developed; one that takes into account not only directional distance but also edge colors. 
	
Finally, while we looked at \textsc{nim} heaps of size one and two, the question of metric dimension of the game on an arbitrary number of heaps remains unanswered.

\bibliography{MD-game-graphs}
\bibliographystyle{abbrv}

\end{document}